\documentclass[a4paper,reqno]{amsart}
\usepackage[initials]{amsrefs}
\usepackage{amsmath,amssymb,amsthm}
\usepackage{hyperref}
\hypersetup{colorlinks=true,linkcolor=blue,citecolor=blue,urlcolor=blue}

\evensidemargin=\oddsidemargin
\numberwithin{equation}{section}
\newtheorem{theorem}{Theorem}[section]
\newtheorem{lemma}[theorem]{Lemma}
\newtheorem{corollary}[theorem]{Corollary}

\theoremstyle{remark}
\newtheorem{remark}[theorem]{Remark}

\newcommand{\A}{\mathcal{A}}
\newcommand{\F}{\mathbb{F}}
\newcommand{\cP}{\mathcal{P}}
\newcommand{\cL}{\mathcal{L}}
\newcommand{\cB}{\mathcal{B}}

\title[Mertens products in arithmetic progressions]
{Mertens products in arithmetic progressions\\ over function fields}

\author[H. Jung]{Hwanyup Jung}
\address{\rm Department of Mathematics Education, Chungbuk National University, Cheongju 28644, Korea}
\email{hyjung@chungbuk.ac.kr}

\begin{document}

\begin{abstract}
We prove a function field analogue of Mertens' formula for Euler products over prime polynomials in
arithmetic progressions in $\F_q[t]$, the counterpart of a formula of Languasco and Zaccagnini over the
integers.  An elementary argument shows that the product over the prime polynomials of degree at most
$n$ equals $e^{-H_n}$, where $H_n$ is the $n$-th harmonic number, up to an exponentially small relative
error.  Combined with the unconditional Riemann hypothesis for Dirichlet $L$-functions, this gives the
main result: the product over a reduced residue class is an explicit Euler-product constant times a
power of the normalizing function $\cL_q(n)=e^{H_n-\gamma}\log q$, the exponent being the reciprocal of
the number of reduced classes, again with an exponentially small error and uniformly for all moduli of
degree at most $n$.  It is $\cL_q(n)$, and not the naive analogue $n\log q$ of $\log x$ under $x=q^{n}$,
that is the right normalization; substituting the latter turns the formula into an expansion in powers
of $1/n$ whose leading coefficients we compute, the leading correction being combinatorial rather than
arithmetic.
\end{abstract}

\maketitle

\section{Introduction and statement of results}\label{sect-1}

Mertens' formula describes the asymptotic behaviour of the Euler product over primes at the edge of
convergence.  In the form relevant to this paper, it states that
\begin{equation}\label{eq:classical-mertens}
\prod_{p\le x}\left(1-\frac1p\right)=\frac{e^{-\gamma}}{\log x}\left(1+o(1)\right),
\end{equation}
where $\gamma$ is Euler's constant.
Quantitative refinements of \eqref{eq:classical-mertens} were obtained by A.~Vinogradov~\cites{Vin62,Vin63}.
A natural refinement in another direction is to restrict the product to primes in a fixed arithmetic
progression: for $(a,m)=1$ one sets
$$
P(x;m,a)=\prod_{\substack{p\le x\\ p\equiv a \bmod m}}\left(1-\frac1p\right).
$$
The asymptotic behaviour of $P(x;m,a)$ for a fixed progression was studied by
Williams~\cite{Williams74} and by Vasil'\-kov\-ska\-ja~\cite{Vasil77}, and explicit estimates were
obtained by Bor\-del\-l\`es~\cite{Bord05}.
Languasco and Zaccagnini~\cite{LZ07} gave a refined treatment which is uniform in the modulus, including a
description of the possible exceptional-zero contribution, together with a simplified Euler product
expression for the leading constant; the constant itself was analysed further in~\cites{LZ09,LZ10}, and
its behaviour as the modulus grows was recently revisited by Keliher and Lee~\cite{KL24}.

The purpose of this paper is to prove the corresponding result over the polynomial ring $\A=\F_q[t]$, and
to make it as precise as the function field setting allows.  We use the non-reciprocal convention of
\eqref{eq:classical-mertens}, in accordance with the arithmetic progression product considered by
Languasco and Zaccagnini~\cite{LZ07}.  This convention is the reciprocal of the Euler product convention
often used in formulations of Mertens' theorem for zeta functions, such as Rosen's global function field
version~\cite{Rosen99}.

\subsection*{Notation}
Let $\A=\F_q[t]$, let $\A^{+}$ denote the set of monic polynomials in $\A$, and let
$\cP\subset\A^{+}$ denote the set of monic irreducible polynomials; we refer to the elements of $\cP$
as \emph{prime polynomials}, in
analogy with rational primes.  For $0\ne f\in\A$ put $|f|=q^{\deg f}$, and set $|0|=0$.  If
$\mathcal S\subset\A$ and $n\ge0$, define
$$
\mathcal S_n=\{f\in\mathcal S:\deg f=n\},\quad
\mathcal S_{\le n}=\{f\in\mathcal S:\deg f\le n\},\quad
\mathcal S_{>n}=\{f\in\mathcal S:\deg f>n\};
$$
thus $\cP_{\le n}$ is the set of prime polynomials of degree at most $n$.  We write
$H_n=\sum_{d=1}^{n}1/d$ for the $n$-th harmonic number.
Throughout, $q$ is a fixed prime power and $\gamma$ denotes Euler's constant.  Implied constants in the
symbols $O(\cdot)$ and $\ll$ are \emph{absolute} unless a subscript indicates otherwise; in particular
they do not depend on $q$, on the modulus $Q$, or on the residue class.

\subsection*{The global formula}
Our starting point is a sharp form of the global Mertens formula over $\A$.  It identifies the correct
normalizing function, and it does so by an entirely elementary argument, with no Tauberian input.

\begin{theorem}\label{thm:global}
For every $n\ge1$, one has 
\begin{equation}\label{eq-1.1}
\prod_{P\in\cP_{\le n}}\left(1-|P|^{-1}\right)=e^{-H_n}\left(1+O\left(\frac{q^{-n/2}}{n}\right)\right),
\end{equation}
with an absolute implied constant.  Consequently
\begin{equation}\label{eq-1.1b}
\prod_{P\in\cP_{\le n}}\left(1-|P|^{-1}\right)
=\frac{\kappa_q}{n\log q}\left(1-\frac1{2n}+O\left(\frac1{n^2}\right)\right),
\end{equation}
where $\kappa_q:=e^{-\gamma}\log q$.
\end{theorem}

Theorem~\ref{thm:global} recovers, in the special case $\A=\F_q[t]$, the global function field Mertens
theorem of Rosen~\cite{Rosen99} (see also~\cite{RosenBook}*{Chapter~5}), which is stated for the
reciprocal product over all prime divisors of $\F_q(t)$ and identifies the leading constant.  What is new
here is that the constant is obtained directly, that the error is exponentially small, and that the
implied constants are absolute; the argument moreover extends to an arbitrary global function field,
see Remark~\ref{rem:general-K}.  

Formula~\eqref{eq-1.1} says that the function field Mertens product is \emph{exactly} $e^{-H_n}$ up to an
exponentially small relative error.  Since $x=q^n$ gives $\log x=n\log q$, the naive analogue of $\log x$
is $n\log q$; by \eqref{eq-1.1b} it is too small, the correct normalizing function exceeding it by a
factor $1+\frac1{2n}+O(n^{-2})$.
We therefore set
\begin{equation}\label{eq-1.L}
\cL_q(n):=e^{H_n-\gamma}\log q=n\log q\,e^{\delta_n},
\qquad\text{where}\quad \delta_n:=H_n-\log n-\gamma .
\end{equation}
By the Euler--Maclaurin expansion,
$\delta_n=\frac1{2n}-\frac1{12n^{2}}+O(n^{-4})$, so that
$\cL_q(n)=n\log q\bigl(1+\frac1{2n}+\frac1{24n^{2}}+O(n^{-3})\bigr)$, while
$\kappa_q/\cL_q(n)=e^{-H_n}$ identically.  Thus \eqref{eq-1.1} reads
$$
\prod_{P\in\cP_{\le n}}\left(1-|P|^{-1}\right)
=\frac{\kappa_q}{\cL_q(n)}\left(1+O\left(\frac{q^{-n/2}}{n}\right)\right),
$$
which is \eqref{eq-1.1b} with $n\log q$ replaced by $\cL_q(n)$ and the error improved from $O(n^{-2})$ to
$O(q^{-n/2}/n)$.  The function $\cL_q(n)$, rather than $n\log q$, is the correct function field substitute
for $\log x$ at this level of precision.

\subsection*{The main theorem}
Let $Q\in\A^{+}$ be nonconstant and put $\Phi(Q)=\#(\A/Q\A)^\times$.  For $A_0\in\A$ with $(A_0,Q)=1$ we
define
\begin{equation}\label{eq-1.2}
P(n;Q,A_0)=\prod_{\substack{P\in\cP_{\le n}\\ P\equiv A_0 \bmod Q}}\left(1-|P|^{-1}\right).
\end{equation}

Our main result is the analogue of Theorem~\ref{thm:global} for the restricted product \eqref{eq-1.2}:
the same normalizing function $\cL_q(n)$ appears, now with exponent $-1/\Phi(Q)$, and the constant of
proportionality is given by an explicit Euler product.

\begin{theorem}\label{thm:main}
Let $Q\in\A^{+}$ be nonconstant and let $A_0\in\A$ satisfy $(A_0,Q)=1$.  Then there is a constant
$C_{Q,A_0}>0$, depending only on $Q$ and on the residue class $A_0\bmod Q$, such that, uniformly for all
$n\ge\deg Q$,
\begin{equation}\label{eq-1.main}
P(n;Q,A_0)=C_{Q,A_0}\,\cL_q(n)^{-1/\Phi(Q)}
\left(1+O\left(\frac{(1+\deg Q)\,q^{-n/2}}{n}\right)\right)
\end{equation}
with an absolute implied constant.  Moreover $C_{Q,A_0}$ is determined by
\begin{equation}\label{eq-1.3}
C_{Q,A_0}^{\Phi(Q)}=\kappa_q\prod_{\substack{P\in\cP\\ P\mid Q}}\left(1-|P|^{-1}\right)^{-1}
\prod_{\substack{P\in\cP\\ P\nmid Q}}\left(1-|P|^{-1}\right)^{\alpha(P;Q,A_0)},
\end{equation}
where, for $P\nmid Q$,
$$
\alpha(P;Q,A_0)=\begin{cases}\Phi(Q)-1, & \text{if } P\equiv A_0 \bmod Q, \\-1, & \text{otherwise.}\end{cases}
$$
The second product in \eqref{eq-1.3} is understood as the limiting logarithmic Euler product, that is, as
the limit of the products truncated at $\deg P\le N$ as $N\to\infty$.
\end{theorem}

The hypothesis $n\ge\deg Q$ of Theorem~\ref{thm:main} can be relaxed; see Remark~\ref{rem:range}.

Substituting the definition \eqref{eq-1.L} of $\cL_q(n)$ into \eqref{eq-1.main} and expanding
$e^{-\delta_n/\Phi(Q)}$ in powers of $n^{-1}$ converts it into an asymptotic expansion in the more
familiar normalization $(n\log q)^{-1/\Phi(Q)}$.

\begin{corollary}\label{cor:expansion}
Under the hypotheses of Theorem~\ref{thm:main}, one has 
\begin{equation}\label{eq-1.exp}
P(n;Q,A_0)=C_{Q,A_0}\,(n\log q)^{-1/\Phi(Q)}
\left(1-\frac{1}{2\Phi(Q)\,n}+\frac{1}{n^{2}}\left(\frac{1}{12\Phi(Q)}+\frac{1}{8\Phi(Q)^{2}}\right)+R_n\right),
\end{equation}
where
$$
R_n\ll\frac1{\Phi(Q)\,n^{3}}+\frac{(1+\deg Q)\,q^{-n/2}}{n}
$$
with an absolute implied constant.  More generally, for every fixed $J\ge1$,
\begin{equation}\label{eq-1.expJ}
P(n;Q,A_0)=C_{Q,A_0}\,(n\log q)^{-1/\Phi(Q)}
\left(\sum_{j=0}^{J}\frac{c_j\bigl(\Phi(Q)^{-1}\bigr)}{n^{j}}+R^{(J)}_n\right),
\end{equation}
where
$$
R^{(J)}_n\ll_J\frac1{\Phi(Q)\,n^{J+1}}+\frac{(1+\deg Q)\,q^{-n/2}}{n},
$$
and where $c_0=1$ while, for $j\ge1$, $c_j$ is a polynomial in one variable with rational coefficients
and without constant term, depending only on $j$ --- in particular independent of $q$, of $Q$ and of
$A_0$ --- so that $c_j\bigl(\Phi(Q)^{-1}\bigr)\ll_J\Phi(Q)^{-1}$.  For instance $c_1(X)=-\frac{X}{2}$ and
$c_2(X)=\frac{X}{12}+\frac{X^{2}}{8}$, so that \eqref{eq-1.exp} is the case $J=2$ of
\eqref{eq-1.expJ}.
\end{corollary}

\begin{remark}\label{rem:visible}
The secondary term $-1/(2\Phi(Q)n)$ in \eqref{eq-1.exp} is not an arithmetic phenomenon: by
Theorem~\ref{thm:main} it is nothing but the discrepancy between the two normalizations $\cL_q(n)$ and
$n\log q$, spread over the $\Phi(Q)$ residue classes.  It thus originates in the Euler--Maclaurin
expansion of $H_n$, and so reflects the integrality of $\deg P$; in particular it carries no information
about $A_0$, depends on $Q$ only through $\Phi(Q)$, and is independent of $q$.

Whether the secondary term is visible at all depends on the modulus: it dominates $R_n$ exactly when
$\Phi(Q)(1+\deg Q)=o(q^{n/2})$, which, since $\Phi(Q)\le q^{\deg Q}$, holds whenever $\deg Q\le\eta n$
with a fixed $\eta<\frac12$.  When instead $\deg Q$ is close to $n$ and $\Phi(Q)$ is of size $q^{n}$, the
whole expansion \eqref{eq-1.exp} is swamped by the exponentially small error and \eqref{eq-1.main} is the
only statement with content.  This is no defect of the method: for such $Q$ even $\cL_q(n)^{-1/\Phi(Q)}$
differs from $1$ by only $O\bigl(\Phi(Q)^{-1}\log(n\log q)\bigr)$, less than the error term of
\eqref{eq-1.main}, so no normalization by a power of $n\log q$ can say more than \eqref{eq-1.main}
itself.
\end{remark}

Since the $n$-dependence in \eqref{eq-1.main} is the same for every residue class, the dependence of
$P(n;Q,A_0)$ on $A_0$ is carried \emph{exactly} by the constant, up to an exponentially small error.

\begin{corollary}\label{cor:ratio}
Let $Q\in\A^{+}$ be nonconstant and let $A_0,A_1$ be coprime to $Q$.  Then, uniformly for
$n\ge\deg Q$, one has
$$
\frac{P(n;Q,A_0)}{P(n;Q,A_1)}=\frac{C_{Q,A_0}}{C_{Q,A_1}}
\left(1+O\left(\frac{(1+\deg Q)\,q^{-n/2}}{n}\right)\right).
$$
Moreover
\begin{equation}\label{eq-1.prodC}
\prod_{\substack{A_0\bmod Q\\ (A_0,Q)=1}}C_{Q,A_0}
=\kappa_q\prod_{\substack{P\in\cP\\ P\mid Q}}\left(1-|P|^{-1}\right)^{-1}.
\end{equation}
\end{corollary}

Theorem~\ref{thm:main} should be compared with the uniform formula of Languasco and
Zaccagnini~\cite{LZ07}*{Theorem~2} over the integers, and \eqref{eq-1.3} with the identities for the
integer constant obtained in~\cite{LZ10}: the shape of \eqref{eq-1.3} is exactly that of
$C(m,a)^{\varphi(m)}=e^{-\gamma}\,\frac{m}{\varphi(m)}\,c(m,a)$, with $e^{-\gamma}$ replaced by $\kappa_q$
and $m/\varphi(m)=\prod_{p\mid m}(1-p^{-1})^{-1}$ by $\prod_{P\mid Q}(1-|P|^{-1})^{-1}$, while
\eqref{eq-1.prodC} is the analogue of the corresponding integer product formula.

The analytic difference between the two settings is decisive.  Classically, uniform asymptotic formulae
must accommodate a possible exceptional zero of a Dirichlet $L$-function, and the error is a power of
$\log x$ rather than of $x$.  Over $\F_q[t]$ Weil's Riemann hypothesis holds unconditionally
(see~\cite{RosenBook}*{Chapter~4}), so the nonprincipal character sums exhibit square-root cancellation
and no exceptional-zero correction can occur; this is what makes the error in \eqref{eq-1.main}
$O(q^{-n/2})$, that is $O(x^{-1/2})$ in the variable $x=q^{n}$, so that the formula is of ``GRH
strength'' unconditionally.  It is also why the secondary term in \eqref{eq-1.exp} is governed entirely
by \eqref{eq-1.1}: the nonprincipal characters contribute nothing at any polynomial scale.  No exact
analogue of that term occurs over the integers, where $\log x$ varies continuously.

\subsection*{Outline of the proof and of the paper}
The proof follows the classical character decomposition.  After taking logarithms in \eqref{eq-1.2}, the
congruence $P\equiv A_0\bmod Q$ is detected by Dirichlet characters modulo $Q$.  The principal character
produces the main term through the global formula \eqref{eq-1.1}.  The nonprincipal characters
give convergent limiting logarithmic products, and their truncation errors are controlled by square-root
cancellation for prime polynomial sums.

Section~\ref{sect-2} proves Theorem~\ref{thm:global}.  Section~\ref{sect-3} collects character
orthogonality, Dirichlet $L$-functions over $\A$, square-root cancellation, and tail estimates for
twisted prime sums.  Section~\ref{sect-4} combines these ingredients to prove Theorem~\ref{thm:main},
Corollary~\ref{cor:expansion} and Corollary~\ref{cor:ratio}, including the Euler product expression and
the positivity of $C_{Q,A_0}$.  Section~\ref{sect-5} gives a numerical illustration over $\F_2[t]$.

\section{The global Mertens formula}\label{sect-2}

We write $\pi_k=\#\cP_k$ for the number of prime polynomials of degree $k$. 
Recall the exact \emph{prime polynomial theorem}: for every $k\ge1$,
\begin{equation}\label{eq-2.0a}
k\,\pi_k=\sum_{d\mid k}\mu(k/d)\,q^{d},
\end{equation}
where $\mu$ is the M\"obius function; see e.g.~\cite{RosenBook}*{Chapter~2}.  This is an \emph{exact}
identity, not merely an asymptotic one, and it is the basis of everything in this section.  We shall also
use repeatedly the elementary bound
\begin{equation}\label{eq-2.0b}
k\,\pi_k\le q^{k}\qquad(k\ge1),
\end{equation}
which holds because the monic irreducible polynomials of degree $k$ have altogether $k\pi_k$ distinct
roots, all lying in $\F_{q^k}$.

We put
$$
\varepsilon_k:=\pi_k\log\bigl(1-q^{-k}\bigr)+\frac1k\qquad(k\ge1),
$$
and we let
$$
\zeta_\A(u)=\sum_{f\in\A^{+}}u^{\deg f}=\prod_{P\in\cP}\bigl(1-u^{\deg P}\bigr)^{-1}=\frac{1}{1-qu}
$$
denote the zeta function of $\A$.

The whole of Theorem~\ref{thm:global} is contained in the following lemma, whose second assertion is the
substitute for a Tauberian argument.

\begin{lemma}\label{lem:eps}
For every $k\ge1$ one has $|\varepsilon_k|\le 5q^{-k/2}/k$.  Moreover
$$
\sum_{k\ge1}\varepsilon_k=0 .
$$
\end{lemma}

\begin{proof}
It is convenient to prove the bound in a form uniform in a real parameter $\sigma\ge1$.  For $\sigma\ge1$ and
$k\ge1$ put
$$
\varepsilon_k(\sigma):=\pi_k\log\bigl(1-q^{-k\sigma}\bigr)+\frac{q^{-k(\sigma-1)}}{k},
$$
so that $\varepsilon_k(1)=\varepsilon_k$.
By \eqref{eq-2.0a} we may write $k\pi_k=q^k+\vartheta_k$ with
$\vartheta_k=\sum_{d\mid k,\ d<k}\mu(k/d)q^{d}$, so that $\vartheta_1=0$ and, for $k\ge2$,
\begin{equation}\label{eq-2.theta}
|\vartheta_k|\le\sum_{1\le d\le k/2}q^{d}\le\frac{q}{q-1}\,q^{k/2}\le 2q^{k/2}.
\end{equation}
For $0<x\le\frac12$ write $\log(1-x)=-x-\rho(x)$, where $\rho(x):=\sum_{m\ge2}\frac{x^m}{m}$; then
$0\le\rho(x)\le\frac{x^2}{2(1-x)}\le x^2$.  Applying this with $x=q^{-k\sigma}\le q^{-1}\le\frac12$, we get
$$
\pi_k\log\bigl(1-q^{-k\sigma}\bigr)
=\frac{q^k+\vartheta_k}{k}\bigl(-q^{-k\sigma}-\rho(q^{-k\sigma})\bigr)
=-\frac{q^{-k(\sigma-1)}}{k}-\frac{\vartheta_k q^{-k\sigma}}{k}
-\frac{(q^k+\vartheta_k)\,\rho(q^{-k\sigma})}{k},
$$
whence
\begin{equation}\label{eq-2.epsformula}
\varepsilon_k(\sigma)=-\frac{\vartheta_k q^{-k\sigma}}{k}-\frac{(q^k+\vartheta_k)\,\rho(q^{-k\sigma})}{k}.
\end{equation}
Since $\sigma\ge1$ we have $q^{-k\sigma}\le q^{-k}$ and $\rho(q^{-k\sigma})\le q^{-2k\sigma}\le q^{-2k}$,
so \eqref{eq-2.theta} gives
$$
|\varepsilon_k(\sigma)|\le\frac{2q^{k/2}q^{-k}}{k}+\frac{(q^{k}+2q^{k/2})q^{-2k}}{k}
\le\frac{2q^{-k/2}}{k}+\frac{3q^{-k}}{k}\le\frac{5q^{-k/2}}{k},
$$
where in the last step we used $q^{-k}\le q^{-k/2}$.  Taking $\sigma=1$ gives the bound.

It remains to prove the identity.  Let $\sigma>1$; both series occurring below then converge absolutely,
their terms being of constant sign.  Grouping the primes in the Euler product for $\zeta_\A$ by degree,
$$
\sum_{k\ge1}\pi_k\log\bigl(1-q^{-k\sigma}\bigr)
=\log\prod_{P\in\cP}\bigl(1-|P|^{-\sigma}\bigr)=-\log\zeta_\A(q^{-\sigma})=\log\bigl(1-q^{1-\sigma}\bigr),
$$
while
$$
\sum_{k\ge1}\frac{q^{-k(\sigma-1)}}{k}=-\log\bigl(1-q^{1-\sigma}\bigr).
$$
Adding the two identities gives
\begin{equation}\label{eq-2.sumzero}
\sum_{k\ge1}\varepsilon_k(\sigma)=0
\end{equation}
for every $\sigma>1$.
For each fixed $k$, formula \eqref{eq-2.epsformula} shows that $\sigma\mapsto\varepsilon_k(\sigma)$ is
continuous on $[1,\infty)$, so $\varepsilon_k(\sigma)\to\varepsilon_k$ as $\sigma\to1^{+}$.  By the bound
just proved, the series in \eqref{eq-2.sumzero} is dominated, uniformly in $\sigma\ge1$, by the convergent
series $\sum_{k\ge 1} 5q^{-k/2}/k$.  Dominated convergence therefore permits passing to the limit in
\eqref{eq-2.sumzero} as $\sigma\to1^{+}$, and yields $\sum_{k\ge1}\varepsilon_k=0$.
\end{proof}

\begin{proof}[Proof of Theorem~\ref{thm:global}]
Grouping the prime polynomials of $\cP_{\le n}$ by degree and using $\sum_{k\ge1}\varepsilon_k=0$,
$$
\log\prod_{P\in\cP_{\le n}}\bigl(1-|P|^{-1}\bigr)
=\sum_{k=1}^{n}\pi_k\log\bigl(1-q^{-k}\bigr)
=\sum_{k=1}^{n}\left(\varepsilon_k-\frac1k\right)
=-H_n+\sum_{k=1}^{n}\varepsilon_k
=-H_n-\sum_{k>n}\varepsilon_k .
$$
By the bound of Lemma~\ref{lem:eps},
$$
\left|\sum_{k>n}\varepsilon_k\right|\le\frac{5}{n+1}\sum_{k>n}q^{-k/2}
=\frac{5}{n+1}\cdot\frac{q^{-n/2}}{q^{1/2}-1}\le\frac{13\,q^{-n/2}}{n},
$$
since $q^{1/2}-1\ge\sqrt2-1$.  Hence
$\log\prod_{P\in\cP_{\le n}}(1-|P|^{-1})=-H_n+O\bigl(q^{-n/2}/n\bigr)$ with an absolute implied constant,
and since this error term is bounded by $13/\sqrt2$, exponentiating gives \eqref{eq-1.1}, again with an
absolute implied constant.

For \eqref{eq-1.1b}, recall the Euler--Maclaurin expansion
$H_n=\log n+\gamma+\frac1{2n}-\frac1{12n^2}+O(n^{-4})$.  Therefore
$$
e^{-H_n}=\frac{e^{-\gamma}}{n}\exp\left(-\frac1{2n}+O\left(\frac1{n^2}\right)\right)
=\frac{e^{-\gamma}}{n}\left(1-\frac1{2n}+O\left(\frac1{n^2}\right)\right)
=\frac{\kappa_q}{n\log q}\left(1-\frac1{2n}+O\left(\frac1{n^2}\right)\right),
$$
and \eqref{eq-1.1b} follows because $q^{-n/2}/n\ll n^{-2}$ for $q\ge2$.
\end{proof}

\begin{remark}\label{rem:general-K}
The argument is not special to $\A=\F_q[t]$.  Let $K$ be a global function field with exact field of
constants $\F_q$, let $v$ run over \emph{all} places of $K$, and put $Nv=q^{\deg v}$,
$\zeta_K(s)=\prod_v(1-Nv^{-s})^{-1}$ and $\rho_K=\operatorname*{Res}_{s=1}\zeta_K(s)$.  Setting
$\varepsilon_k(\sigma)=\pi_K(k)\log(1-q^{-k\sigma})+q^{-k(\sigma-1)}/k$, where $\pi_K(k)$ is the number of
places of degree $k$, the Riemann hypothesis for the curve gives $|k\pi_K(k)-q^{k}|\ll_K q^{k/2}$, and the
computation leading to \eqref{eq-2.epsformula} again yields $|\varepsilon_k(\sigma)|\ll_K q^{-k/2}/k$
uniformly for $\sigma\ge1$.  In place of \eqref{eq-2.sumzero} one now gets
$\sum_{k\ge1}\varepsilon_k(\sigma)=-\log\zeta_K(\sigma)-\log(1-q^{1-\sigma})\to-\log(\rho_K\log q)$ as
$\sigma\to1^{+}$, whence
$$
\prod_{\deg v\le n}\bigl(1-Nv^{-1}\bigr)=\frac{e^{-H_n}}{\rho_K\log q}
\left(1+O_K\left(\frac{q^{-n/2}}{n}\right)\right),
$$
an explicit form of the constant in Rosen's theorem~\cite{Rosen99}.  For $K=\F_q(t)$ one has
$\zeta_K(s)=\bigl((1-q^{-s})(1-q^{1-s})\bigr)^{-1}$ and $\rho_K\log q=(1-q^{-1})^{-1}$; dividing out the
Euler factor of the infinite place, which \eqref{eq-1.1} omits, recovers \eqref{eq-1.1} exactly.
\end{remark}

\section{Characters, $L$-functions and square-root cancellation}\label{sect-3}

In this section we collect the tools needed to control the nonprincipal characters.  Throughout,
$Q\in\A^{+}$ is nonconstant, $\chi$ ranges over the Dirichlet characters modulo $Q$, extended to $\A$ by
$\chi(F)=0$ whenever $(F,Q)\ne1$, and $\chi_0$ denotes the principal character.  We begin with the
standard orthogonality relation, which detects the congruence condition $F\equiv A_0\bmod Q$.

\begin{lemma}[Character orthogonality]\label{lem-orthogonality}
Let $Q\in\A^{+}$ and let $A_0$ be coprime to $Q$.  For every $F\in\A$,
$$
\frac{1}{\Phi(Q)}\sum_{\chi\bmod Q}\overline{\chi}(A_0)\chi(F)=\mathbf{1}_{F\equiv A_0 \bmod Q}.
$$
\end{lemma}

\begin{proof}
If $(F,Q)\ne1$ then $F\not\equiv A_0\bmod Q$, since $(A_0,Q)=1$, and both sides vanish.  If $(F,Q)=1$, the
identity is the standard orthogonality relation for the character group of $(\A/Q\A)^\times$;
see~\cite{RosenBook}*{Chapter~4}.
\end{proof}

Applying Lemma~\ref{lem-orthogonality} after taking logarithms in \eqref{eq-1.2} reduces the proof of
Theorem~\ref{thm:main} to estimating, for each $\chi\bmod Q$, the logarithmic sums
\begin{equation}\label{eq-3.S}
S_n(\chi):=\sum_{P\in\cP_{\le n}}\chi(P)\log\left(1-|P|^{-1}\right).
\end{equation}
For $\chi=\chi_0$ this reduces, up to the finitely many primes dividing $Q$, to the unrestricted sum
governed by Theorem~\ref{thm:global}.  For $\chi\ne\chi_0$ we need square-root cancellation, which we now
establish.

\subsection{Dirichlet $L$-functions and square-root cancellation}\label{sect3-1}

For a Dirichlet character $\chi$ modulo $Q$ define
\begin{equation}\label{eq:L-def}
L(u,\chi)=\sum_{F\in\A^{+}}\chi(F)u^{\deg F}
=\prod_{\substack{P\in\cP\\ P\nmid Q}}\left(1-\chi(P)u^{\deg P}\right)^{-1}.
\end{equation}
If $\chi$ is nonprincipal, then $L(u,\chi)$ is a polynomial of degree at most $\deg Q-1$;
see~\cite{RosenBook}*{Chapter~4}.  Its reciprocal zeros satisfy the Riemann hypothesis for function
fields: apart from possible trivial reciprocal zeros of absolute value $1$ --- which arise from the Euler
factors removed when $\chi$ is imprimitive, and from the reciprocal zero at $u=1$ when the inducing
primitive character is even --- all remaining reciprocal zeros have absolute value exactly $q^{1/2}$;
see~\cite{RosenBook}*{Chapter~4}.  In particular every reciprocal zero has absolute value at most
$q^{1/2}$.

For a nonprincipal character $\chi$ modulo $Q$ set
$$
\Theta(n,\chi)=\sum_{P\in\cP_{\le n}}\chi(P)\deg P .
$$

\begin{lemma}\label{lem-theta}
Let $\chi$ be a nonprincipal Dirichlet character modulo $Q$.  Then, for all $n\ge1$,
$$
\Theta(n,\chi) \ll (\deg Q)\,q^{n/2},
$$
with an absolute implied constant.
\end{lemma}

\begin{proof}
For $m\ge1$ define the twisted von Mangoldt sum
$$
\Psi(m,\chi)=\sum_{F\in\A^{+}_m}\Lambda(F)\chi(F),
$$
where $\Lambda$ is the von Mangoldt function on $\A$, defined by $\Lambda(P^r)=\deg P$ for $P\in\cP$ and
$r\ge1$, and $\Lambda(F)=0$ otherwise.  Taking the logarithmic derivative of the Euler product in
\eqref{eq:L-def} gives
$$
\frac{uL'(u,\chi)}{L(u,\chi)}=\sum_{m\ge 1}\Psi(m,\chi)u^m .
$$
Since $\chi$ is nonprincipal, $L(u,\chi)$ is a polynomial of degree $D_\chi\le\deg Q-1$; write
$L(u,\chi)=\prod_{j=1}^{D_\chi}(1-\alpha_j u)$.  By the Riemann hypothesis for Dirichlet $L$-functions
over function fields, $|\alpha_j|\le q^{1/2}$ for every $j$.  Hence
$$
\frac{uL'(u,\chi)}{L(u,\chi)}=-\sum_{j=1}^{D_\chi}\sum_{m\ge 1}\alpha_j^m u^m ,
$$
and comparing coefficients gives $\Psi(m,\chi)=-\sum_{j=1}^{D_\chi}\alpha_j^m$, so that
\begin{equation}\label{eq-3.psi}
|\Psi(m,\chi)|\le D_\chi\, q^{m/2}\le (\deg Q)\,q^{m/2}.
\end{equation}

Now put $\theta(m,\chi)=\sum_{P\in\cP_m}\chi(P)\deg P$.  The difference between $\Psi(m,\chi)$ and
$\theta(m,\chi)$ comes only from proper prime powers, so $\Psi(m,\chi)=\theta(m,\chi)+R(m,\chi)$ with
$$
R(m,\chi)=\sum_{\substack{r\ge 2,\ P\in\cP\\ r\deg P=m}}\chi(P)^r\deg P .
$$
Since $|\chi(P)|\le1$ and $\sum_{P\in\cP_d}d=d\pi_d\le q^{d}$ by \eqref{eq-2.0b},
$$
|R(m,\chi)|\le\sum_{d\le m/2}\,\sum_{P\in\cP_d}d\le\sum_{d\le m/2}q^{d}\le\frac{q}{q-1}q^{m/2}\le 2q^{m/2}.
$$
Combining this with \eqref{eq-3.psi} gives $|\theta(m,\chi)|\le(\deg Q+2)q^{m/2}\le 3(\deg Q)q^{m/2}$,
because $\deg Q\ge1$.  Finally
$$
|\Theta(n,\chi)|\le\sum_{m\le n}|\theta(m,\chi)|\le 3(\deg Q)\sum_{m\le n}q^{m/2}
\le\frac{3(\deg Q)\,q^{n/2}}{1-q^{-1/2}}\le 11 (\deg Q)\,q^{n/2},
$$
using $q\ge2$.  This proves the lemma, with $11$ as an admissible absolute implied constant.
\end{proof}

\subsection{Tail estimates for twisted prime sums}\label{sect3-2}

We next estimate the tail of the prime polynomial sum twisted by a nonprincipal character.  Combining
Lemma~\ref{lem-theta} with partial summation in the degree aspect controls the contribution of primes of
degree greater than $n$; this is what allows finite twisted logarithmic sums to be replaced by their
limiting values in \S\ref{sect-4}.

\begin{lemma}\label{lem-tail}
Let $\chi$ be a nonprincipal Dirichlet character modulo $Q$.  Then, for all $n\ge 1$,
$$
\sum_{P\in\cP_{>n}}\frac{\chi(P)}{|P|}\ll \frac{\deg Q}{n}\,q^{-n/2},
$$
with an absolute implied constant.
\end{lemma}

\begin{proof}
Put $T(n)=\sum_{P\in\cP_{>n}}\chi(P)|P|^{-1}$.  Grouping by degree,
$$
T(n)=\sum_{k=n+1}^{\infty}q^{-k}\sum_{P\in\cP_k}\chi(P) ,
$$
and since $\Theta(k,\chi)-\Theta(k-1,\chi)=k\sum_{P\in\cP_k}\chi(P)$, we get
$$
T(n)=\sum_{k=n+1}^{\infty}\bigl(\Theta(k,\chi)-\Theta(k-1,\chi)\bigr)b_k,
$$
where $b_k:={q^{-k}}/{k}$.
For $M>n$, summation by parts gives
\begin{equation}\label{eq:tail-sbp}
\sum_{k=n+1}^{M}\bigl(\Theta(k,\chi)-\Theta(k-1,\chi)\bigr)b_k
=\Theta(M,\chi)b_M-\Theta(n,\chi)b_{n+1}+\sum_{k=n+1}^{M-1}\Theta(k,\chi)(b_k-b_{k+1}).
\end{equation}
By Lemma~\ref{lem-theta},
$$
|\Theta(M,\chi)b_M|\ll (\deg Q)\,q^{M/2}\frac{q^{-M}}{M}=(\deg Q)\frac{q^{-M/2}}{M}\longrightarrow0
\qquad(M\to\infty),
$$
so letting $M\to\infty$ in \eqref{eq:tail-sbp} yields
$$
T(n)=-\Theta(n,\chi)b_{n+1}+\sum_{k=n+1}^{\infty}\Theta(k,\chi)(b_k-b_{k+1}).
$$
The first term is bounded by
$$
|\Theta(n,\chi)b_{n+1}|\ll (\deg Q)\,q^{n/2}\frac{q^{-(n+1)}}{n+1}\ll\frac{\deg Q}{n}\,q^{-n/2}.
$$
For the second, $0<b_k-b_{k+1}\le b_k=q^{-k}/k$, so by Lemma~\ref{lem-theta} again
$$
\sum_{k=n+1}^{\infty}|\Theta(k,\chi)|\,|b_k-b_{k+1}|
\ll (\deg Q)\sum_{k=n+1}^{\infty}\frac{q^{-k/2}}{k}
\ll \frac{\deg Q}{n}\,q^{-n/2},
$$
the last sum being dominated by its first term since $q\ge2$.  This proves the lemma.
\end{proof}

\section{Proof of Theorem~\ref{thm:main} and its corollaries}\label{sect-4}

Throughout this section $Q\in\A^{+}$ is nonconstant, $(A_0,Q)=1$, and $n\ge\deg Q$.  Note that
$n\ge\deg Q$ guarantees that every prime divisor of $Q$ lies in $\cP_{\le n}$, and hence also that
$P\in\cP_{>n}$ implies $P\nmid Q$.

\begin{proof}[Proof of Theorem~\ref{thm:main}]
Taking logarithms in \eqref{eq-1.2} and applying Lemma~\ref{lem-orthogonality} --- which is legitimate
because $(A_0,Q)=1$ forces every $P$ with $P\equiv A_0\bmod Q$ to be coprime to $Q$ --- we obtain, with
$S_n(\chi)$ as in \eqref{eq-3.S},
\begin{equation}\label{eq-4.1}
\log P(n;Q,A_0)=\sum_{\substack{P\in\cP_{\le n}\\ P\equiv A_0\bmod Q}}\log\left(1-|P|^{-1}\right)
=\frac{1}{\Phi(Q)}\sum_{\chi\bmod Q}\overline{\chi}(A_0)\,S_n(\chi).
\end{equation}

We first consider the principal character.  Since $\chi_0(P)=1$ for $P\nmid Q$ and $\chi_0(P)=0$ for $P\mid Q$, and since every prime divisor of $Q$
has degree at most $n$,
$$
S_n(\chi_0)=\sum_{P\in\cP_{\le n}}\log\left(1-|P|^{-1}\right)
-\sum_{\substack{P\in\cP\\ P\mid Q}}\log\left(1-|P|^{-1}\right).
$$
By Theorem~\ref{thm:global}, $\sum_{P\in\cP_{\le n}}\log(1-|P|^{-1})=-H_n+O(q^{-n/2}/n)$, so
\begin{equation}\label{eq-4.2}
S_n(\chi_0)=-H_n-\sum_{\substack{P\in\cP\\ P\mid Q}}\log\left(1-|P|^{-1}\right)
+O\left(\frac{q^{-n/2}}{n}\right).
\end{equation}

We next treat the nonprincipal characters.  For $\chi\ne\chi_0$ set
\begin{equation}\label{eq-4.3}
\cB(\chi)=\sum_{\substack{P\in\cP\\ P\nmid Q}}\chi(P)\log\left(1-|P|^{-1}\right),
\end{equation}
the series being ordered by degree.  We claim that it converges, and that
\begin{equation}\label{eq-4.4}
S_n(\chi)=\cB(\chi)+O\left(\frac{(1+\deg Q)\,q^{-n/2}}{n}\right).
\end{equation}
Indeed, since $P\in\cP_{>n}$ implies $P\nmid Q$, we have
$\cB(\chi)-S_n(\chi)=\sum_{P\in\cP_{>n}}\chi(P)\log(1-|P|^{-1})$, whenever the right-hand side converges.
Using $\log(1-z)=-z+O(|z|^2)$ for $|z|\le\frac12$ with $z=|P|^{-1}$,
$$
\sum_{P\in\cP_{>n}}\chi(P)\log\left(1-|P|^{-1}\right)
=-\sum_{P\in\cP_{>n}}\frac{\chi(P)}{|P|}+O\left(\sum_{P\in\cP_{>n}}|P|^{-2}\right).
$$
By Lemma~\ref{lem-tail} the first sum converges and is $O\bigl(\frac{\deg Q}{n}q^{-n/2}\bigr)$, while by
\eqref{eq-2.0b}
$$
\sum_{P\in\cP_{>n}}|P|^{-2}=\sum_{k>n}\pi_k q^{-2k}\le\sum_{k>n}\frac{q^{-k}}{k}\ll\frac{q^{-n}}{n}.
$$
Both bounds tend to $0$, so $\cB(\chi)$ converges, and adding them gives \eqref{eq-4.4}.

Substituting \eqref{eq-4.2} and \eqref{eq-4.4} into \eqref{eq-4.1} and using
$\overline{\chi_0}(A_0)=1$ and $|\overline{\chi}(A_0)|=1$ for $\chi\ne\chi_0$, we get
\begin{equation}\label{eq-4.5}
\log P(n;Q,A_0)=-\frac{H_n}{\Phi(Q)}
+\frac{1}{\Phi(Q)}\left(-\sum_{\substack{P\in\cP\\ P\mid Q}}\log\left(1-|P|^{-1}\right)
+\sum_{\chi\ne\chi_0}\overline{\chi}(A_0)\cB(\chi)\right)+E_n,
\end{equation}
where $E_n$ is defined by \eqref{eq-4.5}: it is the sum over $\chi\bmod Q$ of $\overline{\chi}(A_0)$
multiplied by the error term of \eqref{eq-4.2} (for $\chi=\chi_0$) or of \eqref{eq-4.4} (for
$\chi\ne\chi_0$), divided by $\Phi(Q)$.  Since $|\overline{\chi}(A_0)|=1$ and there are $\Phi(Q)-1$
nonprincipal characters, the triangle inequality gives
$$
|E_n|\le\frac{1}{\Phi(Q)}\cdot O\left(\frac{q^{-n/2}}{n}\right)
+\frac{\Phi(Q)-1}{\Phi(Q)}\cdot O\left(\frac{(1+\deg Q)q^{-n/2}}{n}\right)
= O\left(\frac{(1+\deg Q)\,q^{-n/2}}{n}\right).
$$
The implied constant is absolute.  Note also
that, in the admitted range $\deg Q\le n$ and $q\ge2$,
\begin{equation}\label{eq-4.bdd}
\frac{(1+\deg Q)\,q^{-n/2}}{n}\le\left(1+\frac1n\right)q^{-n/2}\le\sqrt2 ,
\end{equation}
so $E_n$ is bounded by an absolute constant.

Now define $C_{Q,A_0}$ by
\begin{equation}\label{eq-4.6}
\log C_{Q,A_0}=\frac{1}{\Phi(Q)}\left(\log\kappa_q-\sum_{\substack{P\in\cP\\ P\mid Q}}\log\left(1-|P|^{-1}\right)
+\sum_{\chi\ne\chi_0}\overline{\chi}(A_0)\,\cB(\chi)\right).
\end{equation}
The right-hand side is a finite complex number; it is real, and hence $C_{Q,A_0}$ is a well-defined
positive real number, by the identity \eqref{eq-4.8} established below.  Since
$$
\log\cL_q(n)=H_n-\gamma+\log\log q=H_n+\log\kappa_q ,
$$
equation \eqref{eq-4.5} may be rewritten as
$$
\log P(n;Q,A_0)=\log C_{Q,A_0}-\frac{\log\cL_q(n)}{\Phi(Q)}+E_n .
$$
Exponentiating and using $|e^{E_n}-1|\le|E_n|e^{|E_n|}$, which is $O(|E_n|)$ with an absolute implied
constant by \eqref{eq-4.bdd}, gives \eqref{eq-1.main}.

It remains to express $C_{Q,A_0}$ as an Euler product.  Multiplying \eqref{eq-4.6} by $\Phi(Q)$ and inserting the definition \eqref{eq-4.3} of $\cB(\chi)$,
\begin{equation}\label{eq-4.7}
\Phi(Q)\log C_{Q,A_0}=\log\kappa_q-\sum_{\substack{P\in\cP\\ P\mid Q}}\log\left(1-|P|^{-1}\right)
+\sum_{\chi\ne\chi_0}\overline{\chi}(A_0)\sum_{\substack{P\in\cP\\ P\nmid Q}}\chi(P)\log\left(1-|P|^{-1}\right).
\end{equation}
The sum over $\chi$ is finite and each inner series converges, so the order of summation may be
interchanged:
$$
\sum_{\chi\ne\chi_0}\overline{\chi}(A_0)\sum_{\substack{P\in\cP\\ P\nmid Q}}\chi(P)\log\left(1-|P|^{-1}\right)
=\lim_{N\to\infty}\sum_{\substack{P\in\cP_{\le N}\\ P\nmid Q}}
\left(\sum_{\chi\ne\chi_0}\overline{\chi}(A_0)\chi(P)\right)\log\left(1-|P|^{-1}\right).
$$
For $P\nmid Q$, Lemma~\ref{lem-orthogonality} gives
$\sum_{\chi\bmod Q}\overline{\chi}(A_0)\chi(P)=\Phi(Q)\,\mathbf 1_{P\equiv A_0\bmod Q}$, and since the
principal character contributes $1$ we obtain
$\sum_{\chi\ne\chi_0}\overline{\chi}(A_0)\chi(P)=\alpha(P;Q,A_0)$ with $\alpha$ as in the statement of the
theorem.  Consequently
\begin{equation}\label{eq-4.8}
\Phi(Q)\log C_{Q,A_0}=\log\kappa_q-\sum_{\substack{P\in\cP\\ P\mid Q}}\log\left(1-|P|^{-1}\right)
+\lim_{N\to\infty}\sum_{\substack{P\in\cP_{\le N}\\ P\nmid Q}}\alpha(P;Q,A_0)\log\left(1-|P|^{-1}\right),
\end{equation}
in which every term on the right is real and the limit exists.  Hence $\Phi(Q)\log C_{Q,A_0}$ is a real
number, so the quantity $C_{Q,A_0}$ defined by \eqref{eq-4.6} is a well-defined positive real number, as
asserted above.  Exponentiating \eqref{eq-4.8} gives \eqref{eq-1.3}, the second product being interpreted
as the limit of its truncations, as stated.  Formula \eqref{eq-4.8} also shows that $C_{Q,A_0}$ depends
only on $Q$ and on the residue class $A_0\bmod Q$.  This completes the proof.
\end{proof}

\begin{proof}[Proof of Corollary~\ref{cor:expansion}]
By \eqref{eq-1.L}, $\cL_q(n)=(n\log q)e^{\delta_n}$, and
$\delta_n=\frac1{2n}-\frac1{12n^2}+O(n^{-4})$ by Euler--Maclaurin, so
$$
\cL_q(n)^{-1/\Phi(Q)}=(n\log q)^{-1/\Phi(Q)}\exp\left(-\frac{\delta_n}{\Phi(Q)}\right).
$$
Since $-\delta_n/\Phi(Q)=-\frac1{2\Phi(Q) n}+\frac1{12\Phi(Q) n^{2}}+O\bigl(\frac1{\Phi(Q) n^{4}}\bigr)$ and
$0<\delta_n\le\frac1{2n}$, expanding the exponential gives
\begin{equation}\label{eq-4.expand}
\exp\left(-\frac{\delta_n}{\Phi(Q)}\right)
=1-\frac{1}{2\Phi(Q) n}+\frac{1}{n^{2}}\left(\frac{1}{12\Phi(Q)}+\frac{1}{8\Phi(Q)^{2}}\right)
+O\left(\frac{1}{\Phi(Q) n^{3}}\right),
\end{equation}
uniformly in $\Phi(Q)\ge1$: indeed every term of $\sum_{m\ge1}(-\delta_n/\Phi(Q))^{m}/m!$ carries a factor
$\Phi(Q)^{-m}\le\Phi(Q)^{-1}$, and $\delta_n\le\frac1{2n}$ makes the series converge uniformly.  Multiplying
\eqref{eq-4.expand} by the factor $1+O\bigl((1+\deg Q)q^{-n/2}/n\bigr)$ of \eqref{eq-1.main}, and noting
that this factor is bounded by \eqref{eq-4.bdd} while the left-hand side of \eqref{eq-4.expand} lies in
$(0,1]$, produces exactly the two error terms displayed in \eqref{eq-1.exp}.

We stress that the exponentially small error is \emph{not} absorbed into $O(1/(\Phi(Q) n^{3}))$: since
$\Phi(Q)$ may be as large as $q^{\deg Q}$, the quantity $q^{-n/2}$ need not be smaller than
$1/(\Phi(Q) n^{3})$.  See Remark~\ref{rem:visible}.

For \eqref{eq-1.expJ}, use the full Euler--Maclaurin expansion
$\delta_n=\sum_{j=1}^{J}d_j n^{-j}+O_J(n^{-J-1})$ with $d_j\in\mathbb Q$ (and $d_1=\frac12$).  Then
$\exp(-\delta_n/\Phi(Q))=\sum_{j=0}^{J}c_j(\Phi(Q)^{-1})n^{-j}+O_J\bigl(\Phi(Q)^{-1}n^{-J-1}\bigr)$, where
$c_0=1$ and each $c_j$, $j\ge1$, is a polynomial with rational coefficients and no constant term, by the
same term-by-term inspection; the $d_j$, and hence the $c_j$, do not depend on $q$,
on $Q$ or on $A_0$.  Multiplying by the error factor of \eqref{eq-1.main} as before gives
\eqref{eq-1.expJ}.
\end{proof}

\begin{proof}[Proof of Corollary~\ref{cor:ratio}]
For the first assertion, subtract the two instances of the identity
$\log P(n;Q,A_i)=\log C_{Q,A_i}-\Phi(Q)^{-1}\log\cL_q(n)+E_n^{(i)}$ obtained in the proof of
Theorem~\ref{thm:main}; the term $\Phi(Q)^{-1}\log\cL_q(n)$ cancels, leaving
$$
\log\frac{P(n;Q,A_0)}{P(n;Q,A_1)}=\log\frac{C_{Q,A_0}}{C_{Q,A_1}}
+O\left(\frac{(1+\deg Q)\,q^{-n/2}}{n}\right),
$$
and exponentiating is legitimate by \eqref{eq-4.bdd}, exactly as in the proof of \eqref{eq-1.main}.

For \eqref{eq-1.prodC}, sum \eqref{eq-4.6} over the $\Phi(Q)$ reduced residue classes $A_0\bmod Q$.  The
first two terms are independent of $A_0$ and contribute
$\log\kappa_q-\sum_{P\mid Q}\log(1-|P|^{-1})$ in total, while for $\chi\ne\chi_0$ we have
$\sum_{A_0}\overline{\chi}(A_0)=0$, so the third term contributes nothing.  Exponentiating gives
\eqref{eq-1.prodC}.
\end{proof}

\begin{remark}\label{rem:range}
The hypothesis $n\ge\deg Q$ entered the proof twice: it makes every prime divisor of $Q$ lie in
$\cP_{\le n}$, and it yields \eqref{eq-4.bdd}, which licenses the passage from the additive to the
multiplicative form.  Only the first use is needed for the additive statement
$$
\log P(n;Q,A_0)=\log C_{Q,A_0}-\frac{\log\cL_q(n)}{\Phi(Q)}
+O\left(\frac{(1+\deg Q)\,q^{-n/2}}{n}\right),
$$
which therefore holds, with the same proof and an absolute implied constant, for every $Q\in\A^{+}$ all
of whose prime divisors have degree at most $n$, however large $\deg Q$.  The multiplicative form
\eqref{eq-1.main} follows as soon as the error is bounded, that is as soon as $\deg Q\ll n\,q^{n/2}$.
\end{remark}

\section{A numerical illustration}\label{sect-5}

We illustrate the difference between the two normalizations with an exact computation over
$\A=\F_2[t]$.  Take $Q=t^2+t+1$, which is irreducible, so that $\A/Q\A\cong\F_4$ and $\Phi(Q)=3$; the
reduced residue classes are represented by $1$, $t$ and $t+1$.  All prime polynomials of degree at most
$24$ were enumerated by a polynomial sieve, and the counts by degree were checked against the exact
formula \eqref{eq-2.0a}.  Put
$$
X_n=P(n;Q,A_0)\,e^{H_n/\Phi(Q)},\qquad Y_n=P(n;Q,A_0)\,(n\log q)^{1/\Phi(Q)} .
$$
By Theorem~\ref{thm:main}, $X_n$ converges to $C_{Q,A_0}\kappa_q^{-1/\Phi(Q)}$ with an error
$O(q^{-n/2}/n)$, whereas by Corollary~\ref{cor:expansion} the sequence $Y_n$ converges to $C_{Q,A_0}$ only
at the rate $1/(2\Phi(Q)n)$.  The two rates are compared, without reference to the limits, in
Table~\ref{table-1}, which lists the relative increments for $A_0=1$.

\begin{table}[ht]
\renewcommand{\arraystretch}{1.05}
$$
\begin{array}{r|cc|cc}
n & X_n/X_{n-1}-1 & q^{-n/2}/n & Y_n/Y_{n-1}-1 & \dfrac{1}{2\Phi(Q)\,n(n-1)}\\[2pt]
\hline
14 & 1.88\cdot10^{-4} & 5.58\cdot10^{-4} & 1.082\cdot10^{-3} & 9.16\cdot10^{-4}\\
16 & 8.12\cdot10^{-5} & 2.44\cdot10^{-4} & 7.610\cdot10^{-4} & 6.94\cdot10^{-4}\\
18 & 4.77\cdot10^{-5} & 1.09\cdot10^{-4} & 5.822\cdot10^{-4} & 5.45\cdot10^{-4}\\
20 & 1.65\cdot10^{-5} & 4.88\cdot10^{-5} & 4.477\cdot10^{-4} & 4.39\cdot10^{-4}\\
22 & 7.40\cdot10^{-6} & 2.22\cdot10^{-5} & 3.626\cdot10^{-4} & 3.61\cdot10^{-4}\\
24 & 3.99\cdot10^{-6} & 1.02\cdot10^{-5} & 3.017\cdot10^{-4} & 3.02\cdot10^{-4}
\end{array}
$$
\caption{$q=2$, $Q=t^2+t+1$, $A_0=1$, $\Phi(Q)=3$.}\label{table-1}
\end{table}

The increments of $X_n$ decay geometrically, at the rate $q^{-n/2}/n$ predicted by
Theorem~\ref{thm:main}; the third column lists $q^{-n/2}/n$ itself for comparison.  Indeed
$$
\log\frac{X_n}{X_{n-1}}
=\Bigl(\#\{P\in\cP_n:P\equiv A_0\bmod Q\}-\frac{\pi_n}{\Phi(Q)}\Bigr)\log\bigl(1-q^{-n}\bigr)
+\frac{\varepsilon_n}{\Phi(Q)},
$$
so these increments measure two things at once: the fluctuation of the class counts about
$\pi_n/\Phi(Q)$, which is square-root cancellation, and the deviation $\varepsilon_n$ of $\pi_n$ from
$q^{n}/n$.  For the present $Q$ the second is the larger; in four of the six rows the class counts are
exactly $\pi_n/\Phi(Q)$, so the first term vanishes identically.

Here $\Phi(Q)=3$ is small, so by Remark~\ref{rem:visible} the secondary term of \eqref{eq-1.exp} is the
dominant correction for $Y_n$.  Accordingly the fourth column approaches the fifth, the relative
discrepancy falling from $18\%$ at $n=14$ to below $1\%$ from $n=22$ on, which confirms that
$Y_n-Y_{n-1}$ is governed by the secondary term $-1/(2\Phi(Q)n)$ of \eqref{eq-1.exp}.

The computation also gives
$$
C_{Q,1}=1.39647\ldots,\qquad C_{Q,t}=C_{Q,t+1}=0.60957\ldots ,
$$
the equality of the last two reflecting the fact that the substitution $t\mapsto t+1$ is an
automorphism of $\A$ which fixes $Q$, preserves degrees and irreducibility, and interchanges the residue
classes of $t$ and $t+1$.
As a check on \eqref{eq-1.prodC}, the product of the three constants is $0.51889\ldots$, while
$$
\kappa_2\prod_{P\mid Q}\left(1-|P|^{-1}\right)^{-1}=e^{-\gamma}(\log2)\cdot\frac43=0.5188987\ldots ,
$$
the agreement being to the accuracy permitted by truncation at degree $24$.



\begin{bibdiv}
\begin{biblist}

\bib{Bord05}{article}{
   author={Bordell\`es, Olivier},
   title={An explicit Mertens' type inequality for arithmetic progressions},
   journal={JIPAM. J. Inequal. Pure Appl. Math.},
   volume={6},
   date={2005},
   number={3},
   pages={Article 67, 10},
}

\bib{KL24}{article}{
   author={Keliher, Daniel},
   author={Lee, Ethan S.},
   title={On the constants in Mertens' theorems for primes in arithmetic progressions},
   journal={Integers},
   volume={24A},
   date={2024},
   pages={Paper No. A12, 13},
}

\bib{LZ07}{article}{
   author={Languasco, A.},
   author={Zaccagnini, A.},
   title={A note on Mertens' formula for arithmetic progressions},
   journal={J. Number Theory},
   volume={127},
   date={2007},
   number={1},
   pages={37--46},
}

\bib{LZ09}{article}{
   author={Languasco, A.},
   author={Zaccagnini, A.},
   title={On the constant in the Mertens product for arithmetic progressions. II. Numerical values},
   journal={Math. Comp.},
   volume={78},
   date={2009},
   number={265},
   pages={315--326},
}

\bib{LZ10}{article}{
   author={Languasco, A.},
   author={Zaccagnini, A.},
   title={On the constant in the Mertens product for arithmetic progressions. I. Identities},
   journal={Funct. Approx. Comment. Math.},
   volume={42},
   date={2010},
   number={1},
   pages={17--27},
}

\bib{Rosen99}{article}{
   author={Rosen, Michael},
   title={A generalization of Mertens' theorem},
   journal={J. Ramanujan Math. Soc.},
   volume={14},
   date={1999},
   number={1},
   pages={1--19},
}

\bib{RosenBook}{book}{
   author={Rosen, Michael},
   title={Number theory in function fields},
   series={Graduate Texts in Mathematics},
   volume={210},
   publisher={Springer-Verlag, New York},
   date={2002},
   pages={xii+358},
}

\bib{Vasil77}{article}{
   author={Vasil'kovskaja, E. A.},
   title={Mertens' formula for an arithmetic progression},
   language={Russian},
   journal={Ta\v skent. Gos. Univ. Nau\v cn. Trudy},
   date={1977},
   number={548},
   pages={14--17, 139--140},
}

\bib{Vin62}{article}{
   author={Vinogradov, A. I.},
   title={On Mertens' theorem},
   language={Russian},
   journal={Dokl. Akad. Nauk SSSR},
   volume={143},
   date={1962},
   pages={1020--1021},
}

\bib{Vin63}{article}{
   author={Vinogradov, A. I.},
   title={On the remainder in Mertens' formula},
   language={Russian},
   journal={Dokl. Akad. Nauk SSSR},
   volume={148},
   date={1963},
   pages={262--263},
}

\bib{Williams74}{article}{
   author={Williams, Kenneth S.},
   title={Mertens' theorem for arithmetic progressions},
   journal={J. Number Theory},
   volume={6},
   date={1974},
   pages={353--359},
   issn={0022-314X},
}

\end{biblist}
\end{bibdiv}
\end{document}